%% file: main.tex
\documentclass[12pt]{article}
\usepackage[T1]{fontenc}
\usepackage[latin9]{inputenc}
\usepackage[letterpaper]{geometry}
\usepackage{amsthm}
\usepackage{amsmath}
\usepackage{graphicx}
\usepackage{setspace}
\usepackage{amssymb}
\usepackage[subrefformat=parens,labelformat=parens]{subfig}
\usepackage{fullpage,array,natbib,authblk}
\usepackage{hyperref}

\makeatletter

\theoremstyle{plain}
\newtheorem{thm}{Theorem}
  \theoremstyle{plain}
  
  \theoremstyle{plain}
  
  \theoremstyle{plain}
  
\newcommand{\col}{\rm c}
\newcommand{\row}{\rm r}
\long\def\symbolfootnote[#1]#2{\begingroup%
\def\thefootnote{\fnsymbol{footnote}}\footnote[#1]{#2}\endgroup}

\@ifundefined{showcaptionsetup}{}{%
 \PassOptionsToPackage{caption=false}{subfig}}
\usepackage{subfig}
\def\keywords#1{{\vskip4pt
\noindent
\hbox to59.5pt{KEY\enspace WORDS:\quad\hss}\vtop{\advance \hsize by -59.5pt
\leftskip=28pt \rightskip=0pt
\noindent\ignorespaces#1\vskip8pt}}}

\begin{document}

\title{Testing for nodal dependence in relational data matrices}

\author[1]{Alexander Volfovsky}
\author[2]{Peter D. Hoff}
\affil[1]{Department of Statistics, University of Washington}
\affil[2]{Departments of Statistics and Biostatistics, University of Washington}

\maketitle

\symbolfootnote[0]{
This work was partially supported by
NICHD grant 1R01HD067509-01A1. The authors would like to thank
Bailey Fosdick for helpful discussions.}


\input{abstract}
\input{intro_format}
\input{sec2}
\input{power}

\input{extend}
\input{data_format}
\input{discuss}
\input{append}

\bibliographystyle{apalike}
\bibliography{biblio}

\end{document}

%% file: abstract.tex
\begin{abstract}
Relational data are often represented as a square matrix, the entries
of which record the relationships between pairs of objects. Many statistical
methods for the analysis of such data assume some degree of similarity
or dependence between objects in terms of the way they relate to each 
other. However, formal tests for such dependence have not been developed.
We provide a test for such dependence 
using the framework of the matrix normal model, a type of multivariate normal
distribution 
parameterized in terms of row- and column-specific covariance matrices. 
We develop a likelihood ratio test (LRT) for row and column dependence based on the
observation of a single relational data matrix. We obtain
a reference distribution for the LRT statistic, thereby providing an exact 
test for the presence of row or column correlations in a square relational data matrix.
Additionally, we provide extensions of the test to accommodate common features
of such data, such as undefined diagonal entries, a non-zero mean, multiple observations,
and deviations from normality. 
\keywords{Networks; Matrix normal; hypothesis testing; Maximum likelihood}
\end{abstract}

%% file: intro_format.tex
\section{Introduction}

Networks or relational data among $m$ actors, nodes or objects are frequently
presented in the form of an $m\times m$  matrix $Y=\left\{ y_{ij}:1\leq i,j\leq m\right\} $,
where the entry $y_{ij}$ corresponds to a measure of the directed
relationship from object $i$ to object $j$. 
Such data are of interest in a variety of scientific disciplines: 
Sociologists and epidemiologists
gather friendship network data to study social development and health
outcomes among children \citep{fletcher2011you,pollard2010friendship,potter2012estimating,van1999friendship}, 
economists study markets by analyzing networks of business interactions
among companies or countries \citep{westveld2011mixed,lazzarini2001integrating}, and
biologists study gene-gene interaction networks to better understand biological pathways
\citep{bergmann2003similarities,stuart2003gene}. 

Often of interest in
the study of such data is a description of the variation and similarity
among the objects in terms of their relations. 
Similarities among rows and among columns in empirical networks have
long been observed \citep{sampson1968novitiate,leskovec2008statistical}, 
leading to 
the development of statistical tools to summarize
such patterns.
CONCOR (CONvergence of iterated CORrelations)
is an early example of a procedure that partitions the
rows (or columns) of $Y$ into groups based on a summary of the correlations
among the rows (or columns) of $Y$
\citep{white1976social, mcquitty1968clusters}. 
The procedure
yields a ``blockmodel'' of the objects, a representation of the original
data matrix $Y$ by a smaller matrix that identifies relationships
among groups of objects. While this algorithm
is still commonly used \citep{lincoln2004japan,lafosse2006simultaneous},
it suffers from a lack of 
statistical interpretability  \citep{panning1982fitting}, as it is not tied to 
any particular statistical model or inferential goal. 

Several model-based approaches 
presume the existence of a grouping of the objects  such 
that objects
within a group share a common 
distribution for their outgoing relationships. 
This is the notion of stochastic equivalence,  and is the primary 
assumption of stochastic blockmodels, a class of models for which 
the probability of a relationship between two objects depends
only on their individual group memberships \citep{holland1983stochastic,wang1987stochastic,nowicki2001estimation,rohe2011spectral}. \citet{airoldi2008mixed} 
extend the basic blockmodel 
by allowing each object to belong to several groups.
In this model the probability of a relationship between two nodes 
depends on all the group memberships of each object. 
This and other variants of 
stochastic blockmodels
belong to the larger class of latent variable models,  in which the probability
distribution of the relationship between any two objects $i$ and $j$ depends
on unobserved  object-specific latent characteristics $z_i$ and $z_j$
\citep{hoff2002latent}. 
Statistical models of this type all presume some form of 
similarity among the objects in the network. 
However, while such models are widely used and studied, 
no formal test for similarities
among the objects in terms of their relations has been proposed. 

Many statistical  methods for valued or continuous relational data
are developed in the context of normal statistical models. 
These include, for example, the widely-used 
social relations model  
\citep{kenny1984social,li_loken_2002} 
and covariance models for multivariate relational data
\citep{li_2006,westveld_hoff_2011,hoff2011separable}. 
Additionally, statistical models for binary and ordinal relational data
can be based on latent normal random variables via 
probit or  other link functions 
\citep{hoff_2005,hoff_2008_nips}. 
In this article we propose
a novel approach to testing
for similarities between  objects 
in terms of the 
row and  column correlation parameters of
the matrix normal model. The matrix normal
model consists of the multivariate normal distributions that have a 
Kronecker-structured covariance matrix \citep{dawid1981some}.  
Specifically,  we say that an $m\times m$ random matrix $Y$ has the 
mean-zero matrix normal distribution $N_{m\times m}(0,\Sigma_{\row},\Sigma_{\col})$ if 
${\rm vec}\left(Y\right)\sim N_{m^{2}}\left(0,\Sigma_{\col}\otimes\Sigma_{\row}\right) $
where ``vec'' is the vectorization operator and 
``$\otimes$'' denotes the Kronecker product.
Under this distribution, 
the covariance between two relations $y_{ij}$ and $y_{kl}$
is given by ${\rm cov}\left(y_{ij},y_{kl}\right)=\Sigma_{{\row},ik}\Sigma_{{\col},jl}$. 
Furthermore,  it is straightforward to show that 
\begin{align*}
E\left[YY^{t}\right] =  \Sigma_{\row}{\rm tr}\left(\Sigma_{\col}\right)\ \text{and}\ 
E\left[Y^{t}Y\right] =  \Sigma_{\col}{\rm tr}\left(\Sigma_{\row}\right).
\end{align*}
These 
identities suggest the interpretation of $\Sigma_{\row}$ and $\Sigma_{\col}$ 
as the covariance of the objects as senders of ties and as receivers of ties, 
respectively. 
In this article, 
we evaluate evidence for similarities between objects by testing 
for non-zero correlations in this matrix normal model. 
Specifically, we develop a 
test of \begin{align*}
H_{0}:  \  ( \Sigma_{\row},\Sigma_{\col} ) \in\mathcal{D}_{+}^{m} \times  \mathcal{D}_{+}^{m}  \ \ \text{versus} \
H_{1}:  \  (\Sigma_{\row},\Sigma_{\col}) \in (\mathcal{S}_{+}^{m} \times  
   \mathcal{S}_{+}^{m} ) \backslash( \mathcal{D}_{+}^{m}
  \times \mathcal{D}_{+}^{m} ) \end{align*}
where $\mathcal{D}_{+}^{m}$ is the set of $m\times m$ diagonal matrices
with positive entries and $\mathcal{S}_{+}^{m}$ is the set of $m\times m$
positive definite symmetric matrices. Model $H_{0}$, which we call
the Kronecker variance model, represents heteroscedasticity among
the rows and the columns while still maintaining their independence. Model
$H_{1}$, which we call the full Kronecker covariance model, allows
for correlations between all of the rows and all the of columns. 
Rejection of the null
of zero correlation would support further inference via
a model that allowed for similarities
among the objects, such as a stochastic blockmodel,
some other latent variable model or the matrix normal model.
Acceptance of the null would caution against
fitting such a model in order to avoid spurious inferences. 

This goal of evaluating  the evidence for row or column correlation
 is in contrast to that of the existing 
testing literature for matrix normal distributions. 
This literature  has 
focused on an  
evaluation of  the null hypothesis that  ${\rm cov}({\rm vec}(Y)) =\Sigma_{\col}\otimes\Sigma_{\row}$
(our $H_{1}$) against an 
unstructured alternative
\citep{roy2005implementation,mitchell2006likelihood,lu2005likelihood,srivastava2008models}.
The tests proposed in
this literature are likelihood ratio tests that,
 in the case of an $m\times m$ square matrix $Y$, 
 require at 
least $n>m^{2}$ replications
to estimate the covariance under the fully
unstructured model. 
Such tests are 
not applicable to most relational datasets, 
which 
typically consist of 
at most
a few observed relational matrices. 

In the next section we derive a hypothesis test of 
$H_0$ versus $H_1$ in the context of the matrix normal model. 
We show that given a single observed relational matrix $Y$, 
the likelihood is bounded under
both $H_{0}$ and $H_{1}$ and so a likelihood ratio test of $H_{0}$
against $H_{1}$ can be constructed. 
We further show how the 
null distribution of the test statistic can be approximated 
with an arbitrarily high  precision via a Monte Carlo procedure. 
In Section \ref{sec:ellipt} we extend these results to the general
class of matrix variate elliptically contoured distributions.
The power of the test in several different situations is evaluated in 
Section \ref{sec:Power-simulations}. 

Although the development of our testing procedure is 
based on the  
mean-zero matrix normal distribution, 
it is straightforward to extend the test to 
several other scenarios commonly  encountered in the study of relational data, 
including 
missing diagonal entries, non-zero
mean structure, multiple heteroscedastic
observations and binary networks. 
These extensions and two data examples are discussed in Section \ref{sec:Extensions-and-data}. 
 A discussion follows in Section 5.  

%% file: sec2.tex
\section{\label{sec:Likelihood-ratio-test}Likelihood ratio test}

In this section we propose a likelihood ratio test (LRT) for evaluating the 
presence of correlations among the rows and correlations
among the columns of a square matrix. The data matrix $Y$ is modeled
as a draw from a mean zero matrix normal distribution $N_{m\times m}\left(0,\Sigma_{\row},\Sigma_{\col}\right)$.
The parameter space under the null hypothesis $H_{0}$ is $\Theta_0=\mathcal{D}_{+}^{m} \times  \mathcal{D}_{+}^{m}$,
the space of all pairs  of diagonal $m\times m$ matrices with positive entries. Under
the alternative $H_{1}$, the parameter space is $\Theta_1=(\mathcal{S}_{+}^{m} \times  \mathcal{S}_{+}^{m} )\backslash (\mathcal{D}_{+}^{m} \times  \mathcal{D}_{+}^{m}  )$,
the collection of all pairs  of positive definite matrices of dimension $m$ 
for which at least one is not diagonal. 
To derive the LRT statistic, 
we first obtain the maximum likelihood estimates (MLEs) under  
the unrestricted parameter space 
$\Theta= \Theta_0 \cup \Theta_1$ 
 and under the null parameter space $\Theta_0$. 
From these MLEs, we construct 
several equivalent forms of the 
LRT
statistic. 
While the null distribution of the test statistic is not available 
in closed form, the statistic is invariant under 
diagonal rescalings of the data matrix $Y$, implying that 
the distribution of the statistic 
is constant as a function of $(\Sigma_{\row},\Sigma_{\col} ) \in \Theta_0$. 
This fact allows us to obtain null distributions and $p$-values 
via Monte Carlo simulation. 

\subsection{Maximum likelihood estimates}\label{sub:MLE}
The density of a mean zero matrix normal  distribution $N_{m\times m}(0,\Sigma_{\row},\Sigma_{\col})$
is given by\begin{eqnarray*}
p\left(Y|\Sigma_{\row},\Sigma_{\col}\right) & = & 
\left(2\pi\right)^{-m^{2}/2}  
\left|\Sigma_{\col}\otimes\Sigma_{\row}\right|^{-1/2}
\exp (  -\tfrac{1}{2}{\rm tr}\left(\Sigma_{\row}^{-1}Y\Sigma_{\col}^{-1}Y^{t}\right)),\end{eqnarray*}
where ``tr'' is the matrix trace and ``$\otimes$'' is the Kronecker product. 
Throughout 
the article we will write $l\left(\Sigma_{\row},\Sigma_{\col};Y\right)$
as minus two times the log likelihood minus $m^2\log 2\pi$, hereafter referred to as the
scaled log likelihood:
\begin{eqnarray}
l\left(\Sigma_{\row},\Sigma_{\col};Y\right) & = & -2\log p\left(Y|\Sigma_{\row},\Sigma_{\col}\right)  - m^2 \log 2 \pi  ={\rm tr}\left[\Sigma_{\row}^{-1}Y\Sigma_{\col}^{-1}Y^{t}\right]-\log\left|\Sigma_{\col}^{-1}\otimes\Sigma_{\row}^{-1}\right|. 
\label{eq:ld}\end{eqnarray}
We will 
state all the results in this paper in terms of $l\left(\Sigma_{\row},\Sigma_{\col};Y\right)$.
For example, an MLE will be a  minimizer
of  $l\left(\Sigma_{\row},\Sigma_{\col};Y\right)$ in $(\Sigma_{\row},\Sigma_{\col})$. 
The following result implies that if $Y$ is a draw from an 
absolutely continuous distribution on $\mathbb R^{m\times m}$, then the 
scaled likelihood is bounded from below, and achieves this bound on a set of 
nonunique MLEs:
\begin{thm}
\label{thm:alt_not_unique}  
If $Y$ is full rank then  $l( \Sigma_{\row},\Sigma_{\col} ; Y)  
 \geq  m^2 + m\log |YY^t/m| $ for all 
$(\Sigma_{\row},\Sigma_{\col} )\in \Theta$, with equality 
if 
$\Sigma_{\row} =  Y{\Sigma}_{\col}^{-1}Y^{t}/m$, or  equivalently, 
 $\Sigma_{\col}= Y^{t}{\Sigma}_{\row}^{-1}Y/m$. 
\end{thm}

\begin{proof}
We first look for MLEs at the critical points of 
$l\left(\Sigma_{\row},\Sigma_{\col};Y\right)$. 
Setting derivatives of $l$ to zero indicates that critical points satisfy
\begin{eqnarray}
\hat \Sigma_{\row} & = & Y\hat \Sigma_{\col}^{-1}Y^{t}/m\label{eq:full_flip}\\
\hat \Sigma_{\col} & = & Y^{t}\hat \Sigma_{\row}^{-1}Y/m.\label{eq:full_flip2}\end{eqnarray}
Note that these equations are redundant: If $(\hat \Sigma_{\row},
  \hat \Sigma_{\col}) $ satisfy Equation \ref{eq:full_flip}, 
then these values satisfy Equation \ref{eq:full_flip2} as well. 
The value of the scaled log likelihood at such a critical point is 
\begin{eqnarray}
l\left(Y\hat \Sigma_{\col}^{-1}Y^{t}/m,\hat \Sigma_{\col};Y\right) & = & {\rm tr}\left[\left(Y\hat\Sigma_{\col}^{-1}Y^{t}/m\right)^{-1}Y\hat\Sigma_{\col}^{-1}Y^{t}\right]+\log\left|\hat\Sigma_{\col}\otimes Y\hat\Sigma_{\col}^{-1}Y^{t}/m\right|\nonumber \\
 & = & m{\rm tr}\left[Y^{-t}\hat\Sigma_{\col}Y^{-1}Y\hat\Sigma_{\col}^{-1}Y^{t}\right]+m\log\left|\hat\Sigma_{\col}\right|-m\log\left|\hat\Sigma_{\col}\right|+m\log\left|YY^{t}/m\right|\nonumber\\
 & = & m{\rm tr}\left[I\right]+m\log\left|YY^{t}/m\right|.\label{eq:finite_lik} \end{eqnarray}
In the second and third lines, $Y^{-1}$ exists and $|YY^{t}/m|>0$ 
since $Y$ is
square 
and
full rank. 
Now we compare the scaled log likelihood at a critical point 
to its value at any 
other point. 
\begin{align}
l\left(\Sigma_{\row},\Sigma_{\col};Y\right)-l(Y\hat{\Sigma}_{\col}^{-1}Y^{t}/m,\hat{\Sigma}_{\col};Y) & =  
\text{tr}\left[\Sigma_{\row}^{-1}Y\Sigma_{\col}^{-1}Y^{t}\right]   
 + m\log\left(\left|\Sigma_{\row}\right|\left|\Sigma_{\col}\right|\right) 
 -m^{2}  -m\log\left|YY^{t}/m\right|  \\
 &= m^{2}\left[\frac{1}{m}\mbox{tr}\left[\Sigma_{\row}^{-1}Y\Sigma_{\col}^{-1}Y^{t}/m\right]-\frac{1}{m}\log\left|\Sigma_{\row}^{-1}Y\Sigma_{\col}^{-1}Y^{t}/m\right|-1\right]\nonumber
\end{align}
The first equality is a simple combination of equations \ref{eq:ld} and \ref{eq:finite_lik}. The second equality is a rearrangement of terms that combines all the determinant in the log terms.
This difference can be written as $m^2 (a-\log g  -1 )$, where 
$a$ is the arithmetic mean and $g$ is the geometric mean of
the eigenvalues of $(\Sigma_{\row}^{-1/2}Y\Sigma_{\col}^{-1}Y^{t}\Sigma_{\row}^{-1/2})/m$.
To complete the proof we show that $a-\log g-1\geq0$. Consider $f\left(x\right)=x-1-\log x$
and its first and second derivatives with respect to $x$: $f^{\prime}\left(x\right)=1-\frac{1}{x}$,
and $f^{\prime\prime}\left(x\right)=\frac{1}{x^{2}}$. The second
derivative is positive at the critical point $x=1$, so $f\left(1\right)=0$
is a global minimum of the function. Thus $x-\log x-1\geq0$. Now
let $\lambda_1,\dots,\lambda_m$ be the eigenvalues of $(\Sigma_{\row}^{-1/2}Y\Sigma_{\col}^{-1}Y^{t}\Sigma_{\row}^{-1/2})/m$ and so $a=\frac{1}{m}\sum \lambda_{i}$ and $g=\left(\prod \lambda_{i}\right)^{1/m}$. We then have
\[
a\geq\log a+1=\log\left(\frac{1}{m}\sum x_{i}\right)+1\geq\log\left(\left(\prod x_{i}\right)^{1/m}\right)+1=\log g +1\]
as $a\geq g$ since $\lambda_{i}\geq0$ $\forall i$.
Since $a-1-\log g\geq 0$ we have the desired result. 
\end{proof}
Note that the MLE is not unique,
nor is the MLE of $\Sigma_{\col} \otimes \Sigma_{\row}$.  
For example.
$I\otimes YY^t/m$ is an MLE of $\Sigma_{\col} \otimes \Sigma_{\row}$, as is
$Y^tY \otimes I /m$.
Moreover, there is an MLE for each $\Sigma_{\row} \in \mathcal S_{+}^m$ given 
by $( \Sigma_{\row} ,   Y^{t}{\Sigma}_{\row}^{-1}Y/m)$, and similarly 
there is an MLE for each  $\Sigma_{\col} \in \mathcal S_{+}^m$  
given by $( Y \Sigma_{\col}^{-1} Y^t/m ,  \Sigma_{\col} )$.

Theorem \ref{thm:alt_not_unique} also implies that the likelihood is bounded under the null.  Unlike the unrestricted case, the MLE under the null 
is unique up to scalar multiplication: 
\begin{thm}
\label{lem:null_unique} If $Y$ is full rank then
the MLE $\hat{D}_{\col}\otimes\hat{D}_{\row}$ under
$H_0$ is unique,
while $\hat{D}_{\row}$ and $\hat{D}_{\col}$ are unique up to a multiplication and division by the same positive scalar.
\end{thm}
A proof is given in the Appendix. 
To find the MLE under the null model,
we obtain the derivatives of the scaled log likelihood $l$ with
respect to $(\Sigma_{\row}, \Sigma_{\col}) \in \Theta_0$.
For notational convenience, we will refer to 
diagonal versions of $\Sigma_{\row}$ and $\Sigma_{\col}$
as $D_{\row}$ and $D_{\col}$ respectively. 
Setting these
derivatives equal to zero, we establish that the critical points of
$l$ must satisfy \begin{eqnarray}
D_{\row} & = & YD_{\col}^{-1}Y^{t}\circ I/m\label{eq:diag_flip}\\
D_{\col} & = & Y^{t}D_{\row}^{-1}Y\circ I/m\label{eq:diag_flip2}\end{eqnarray}
where {}``$\circ$'' is the Hadamard product. 
The MLE can be found by
iteratively solving equations \eqref{eq:diag_flip} and \eqref{eq:diag_flip2}.
This procedure can be seen as a type of block coordinate descent algorithm, 
decreasing $l$ at each iteration  \citep{tseng2001convergence}.

\subsection{Likelihood ratio test statistic and null distribution}\label{sec:lrts}
Since the scaled log likelihood is bounded below, we are able to obtain
a likelihood ratio statistic that is finite with probability 1 
when $Y$ is sampled from an absolutely continuous distribution on $\mathbb R^{m\times m}$. 
As usual, a likelihood ratio test statistic
can be obtained from the ratio of 
the unrestricted maximized likelihood to the likelihood maximized under the 
null.  We take our test statistic  to be 
\[ T(Y) =     l( \hat D_{\row}, \hat D_{\col} ; Y ) - 
l( \hat \Sigma_{\row}, \hat \Sigma_{\col} ; Y ),   \]
where $( \hat \Sigma_{\row}, \hat \Sigma_{\col})$ is any unrestricted MLE 
and $( \hat D_{\row}, \hat D_{\col} )$ is the MLE under $\Theta_0$.
Since the scaled log likelihood $l$ is minus two times the likelihood, 
our statistic is a monotonically increasing function of the likelihood 
ratio.

In  Theorem 1  we showed that 
$l( \hat{\Sigma}_{\row}, \hat{\Sigma}_{\col}; Y) = 
    m^2+ m\log|YY^t/m| $
for any unrestricted MLE  $(\hat{\Sigma}_{\row}, \hat{\Sigma}_{\col})$. 
Similarly, letting $(\hat{D}_{\row},\hat D_{\col})$ be an MLE under $H_0$, we have 
\begin{align*}
l( \hat{D}_{\row},\hat{D}_{\col};Y)&={\rm tr}\left[Y^t\hat{D}_{\row}^{-1}Y\hat{D}_{\col}^{-1}\right]+\log\left|\hat{D}_{\col}\otimes\hat{D}_{\row}\right|\\
&={\rm tr}\left[(Y^t\hat{D}_{\row}^{-1}Y)(Y^{t}\hat{D}_{\row}^{-1}Y/m\circ I)^{-1}\right]+\log\left|\hat{D}_{\col}\otimes\hat{D}_{\row}\right|\\
&=\sum_i (Y^t \hat{D}_{\row}^{-1}Y)_{ii}(Y^{t}\hat{D}_{\row}^{-1}Y/m\circ I)^{-1}_{ii}+\log\left|\hat{D}_{\col}\otimes\hat{D}_{\row}\right|\\
&=m\sum_i (Y^t \hat{D}_{\row}^{-1}Y)_{ii}/(Y^{t}\hat{D}_{\row}^{-1}Y)_{ii}+\log\left|\hat{D}_{\col}\otimes\hat{D}_{\row}\right|
=m^2 + \log\left|\hat{D}_{\col}\otimes\hat{D}_{\row}\right|,
\end{align*}
where the second equality stems from MLE satisfying $\hat{D}_{\col}=Y^{t}\hat{D}_{\row}^{-1}Y/m\circ I$ (Equation \eqref{eq:diag_flip2}). 
The third equality relies on the following identity for traces: 
for a diagonal matrix $A$ and unstructured matrix $B$ of the same dimension,
${\rm tr}\left[AB\right]=\sum_i A_{ii}B_{ii}$. The final line
is due to the following identity for Hadamard products: if $I$ is the identity matrix and $B$ 
is an unstructured matrix, then $(B\circ I)^{-1}_{ii}=1/B_{ii}$.

The maximized likelihoods under the null and alternative give
\begin{align}
T\left(Y\right)&= 
\log\left|\hat{D}_{\col}\otimes\hat{D}_{\row}\right|-m\log
  \left|YY^{t}/m\right| \nonumber \\
 &= m\left(\log\left|\hat{D}_{\col}\right|+\log\left|\hat{D}_{\row}\right|-\log\left|YY^{t}/m\right|\right).\label{eq:test_clean}
\end{align}

Since no closed form solution exists for $\hat{D}_{\row}$
it is not clear how to obtain the null distributions of $T$ in closed form. 
However,
it is possible to simulate from the null distribution of $T(Y)$, 
as the distribution of the test statistic is the same for all elements of the null hypothesis. 
To see this, we show that the test statistic itself is invariant under
left and right transformations of the data by positive diagonal matrices.
 Let $\tilde{Y}=D_{1}YD_{2}$ for positive diagonal
matrices $D_{1}$ and $D_{2}$. 
Since the MLE $\hat{D}_{\col}\otimes\hat{D}_{\row}$
is unique it is an equivariant function of any matrix $Y$ with respect
to left and right multiplication by diagonal matrices (see \citet{eaton1983multivariate}
Prop 7.11). In particular, writing $\hat \theta \left(Y\right)$ for the estimate
of $\hat{D}_{\col}\otimes\hat{D}_{\row}$ based on a data
matrix $Y$, we have \begin{eqnarray*}
\hat \theta\left(\tilde{Y}\right) & = & \left(D_{2}^{1/2}\otimes D_{1}^{1/2}\right)\hat \theta \left(Y\right)\left(D_{2}^{1/2}\otimes D_{1}^{1/2}\right).\end{eqnarray*}
Since the determinant is a multiplicative map, we can write the
determinant of the above as \begin{eqnarray*}
\left|\hat \theta\left(\tilde{Y}\right)\right| & = & \left|\left(D_{2}\otimes D_{1}\right)\right|\left|\hat \theta\left(Y\right)\right|.\end{eqnarray*}
Using the above, the $T\left(\tilde{Y}\right)$ can be written as
in Equation \eqref{eq:test_clean}: 
\begin{eqnarray*}
T\left(\tilde{Y}\right) & = & m \log\left|\hat \theta\left(\tilde{Y}\right)\right|-m\log\left|\tilde{Y}\tilde{Y}^{t}/m\right|\\
 & = &  m \log\left|\hat \theta\left(Y\right)\right|+ m\log\left|D_{2}\otimes D_{1}\right|
    -m\log\left|YY^{t}/m\right|- m \log\left|D_{2}\otimes D_{1}\right|=T\left(Y\right).\end{eqnarray*}
Since for a matrix normal random variable 
$Y\sim N_{m\times m}\left(0,\Sigma_1,\Sigma_2\right)$ we have 
$Y\overset{d}{=}\Sigma_{1}^{1/2}Y_{0}\Sigma_{2}^{1/2}$ for 
$Y_0\sim N_{m\times m}\left(0,I,I\right)$, 
the above argument implies that $T\left(Y\right)\overset{d}{=}T\left(Y_{0}\right)$
under the null. 
Therefore,  the null distribution of 
$T$ can be approximated via Monte Carlo simulation of 
$Y$ from any distribution in $H_0$.  For example, a Monte Carlo approximation 
to the $q^{\rm th}$ quantile, $T_q$ can be obtained from the following algorithm:
\begin{enumerate}
\item Simulate $Y_{0}^{1},\dots,Y_{0}^{S}  \sim  \text{i.i.d} \ N_{m\times m}\left(0,I,I\right)$; 
\item Let $\hat{T}_q=\min\{T(Y_0^q):\sum_{s=1}^S 1[T(Y_0^q)\geq T(Y_0^s)]/S\geq q\}$.
\end{enumerate}


\subsection{\label{sec:ellipt}Matrix variate elliptically contoured distributions}

The results of the previous subsection are immediately extendable
to the general class of matrix variate elliptically contoured distributions.
In this section we show that under minor regularity
conditions on the distributions, the likelihood for a matrix variate
elliptically contoured distribution is bounded when the matrix normal
distribution is bounded. We provide the form of an MLE for the general
class of mean zero square matrix variate elliptically contoured distributions
and demonstrate that the likelihood ratio test between $H_{0}$ and
$H_{1}$ has the same form as in Equation \eqref{eq:test_clean}. 

We use the notation of \citet{gupta1994new} for the matrix variate
elliptically contoured distribution. We say that $Y$ has a mean zero
square matrix variate elliptically countoured distribution and write
$Y\sim E_{m\times m}\left(0,\Sigma_{\row},\Sigma_{\col},h\right)$
if its density has the form
\begin{align}
f_{Y}\left(Y\right)  =  \frac{1}{\left|\Sigma_{\row}\right|^{\frac{m}{2}}\left|\Sigma_{\col}\right|^{\frac{m}{2}}}h\left({\rm tr}\left[Y^{t}\Sigma_{\row}^{-1}Y\Sigma_{\col}^{-1}\right]\right)
 =  \frac{1}{\left|\Sigma_{\row}\right|^{\frac{m}{2}}\left|\Sigma_{\col}\right|^{\frac{m}{2}}}h\left({\rm vec}\left(Y\right)^{t}\left(\Sigma_{\col}^{-1}\otimes\Sigma_{\row}^{-1}\right){\rm vec}\left(Y\right)\right).\label{eq:h_vec}\end{align}
For $h\left(w\right)=\left(2\pi\right)^{m^{2}/2}\exp\left(-\frac{1}{2}w\right)$,
$Y$ is a mean zero matrix variate normal variable. 
\citet{gupta1994new} show that
if $Y\sim E\left(0,\Sigma_{\row},\Sigma_{\col},h\right)$ then
$AYB\sim E\left(0,A\Sigma_{\row}A^{t},B\Sigma_{\col}B^{t},h\right)$
and if second moments exist, then ${\rm cov}\left({\rm vec}\left(Y\right)\right)=c_h\left(\Sigma_{\col}\otimes\Sigma_{\row}\right)$.
In \cite{gupta1995some}, the authors showed that when $\Sigma_{\row}$
and $h$ are known, the MLE of $\Sigma_{\col}$ is proportional
to the MLE of $\Sigma_{\col}$ under normality, but they
do not provide results for the boundedness of the likelihood or existence
of MLEs for the case where only $h$ is known. 
To find the form of the MLE in this case we state
a simplified version of Theorem 1 of \citet{anderson1986maximum}:
\newtheorem*{thmof}{Theorem 1 of \citet{anderson1986maximum}}
\newtheorem*{thmstr}{Theorem}
\begin{thmstr} Let $\Omega$ be a set
in the space of $\mathcal{S}_{+}^{m^{2}}$, such that if $V\in\Omega$
then $cV\in\Omega$ $\forall c>0$ (that is $\Omega$ is a cone).
Suppose $h$ is such that $h\left(y^{t}y\right)$ is a density in
$\mathbb{R}^{m^{2}}$ and $x^{m^{2}/2}h\left(x\right)$ has a finite
positive maximum $x_{h}$. Suppose that on the basis of an observation
$y$ from $\left|V\right|^{-1/2}h\left(y^{t}V^{-1}y\right)$ an MLE
under normality $\tilde{V}\in\Omega$ exists and $\tilde{V}>0$ with
probability 1. Then an MLE for $h$ is $\hat{V}=\left(m^{2}/x_{h}\right)\tilde{V}$
and the maximum of the likelihood is $\left|\hat{V}\right|^{-1/2}h\left(x_{h}\right)$. 
\end{thmstr}
In the previous subsection we proved that for the mean zero square matrix
normal distribution, the likelihood is bounded for a single observation.
A direct application of the above theorem with $\Omega=\mathcal{S}_{+}^{m}\times\mathcal{S}_{+}^{m}\subset\mathcal{S}_{+}^{m^{2}}$
and $y={\rm vec}\left(Y\right)$ proves that for the likelihood of
a generic matrix variate elliptically contoured distribution with
$h$ defined as in \ref{eq:h_vec}, the likelihood is bounded 
and the MLE of ${\rm cov}({\rm vec}(Y))$ is proportional
to the MLE under normality.
Clearly the theorem hold for a smaller space $\omega=\mathcal{D}_{+}^{m}\times\mathcal{D}_{+}^{m}\subset\mathcal{S}_{+}^{m^{2}}$
as well, and thus the likelihood ratio statistic can be constructed as follows
\begin{align}
T\left(Y\right) & = 2\log\left[\left|\hat{V}_{\Omega}\right|^{-1/2}h\left(x_{h}\right)\right]-2\log\left[\left|\hat{V}_{\omega}\right|^{-1/2}h\left(x_{h}\right)\right]\nonumber \\
 & = \log\left|\hat{V}_{\omega}\right|-\log\left|\hat{V}_{\Omega}\right|
 =\log|\tilde{V}_\omega|-\log|\tilde{V}_\Omega|\label{eq:new_test}\end{align}
where $\tilde{V}_{\omega}$ and $\tilde{V}_{\Omega}$ are the MLEs
under normality that were previously derived.
Equation \eqref{eq:new_test} is identical to the original form of
the test (\emph{e.g.} Equation \eqref{eq:test_clean}). As such, to conduct the test
for any elliptically contoured distribtuion, we can
construct a reference distribution for the null based on a
mean zero matrix variate distribution. 

%% file: power.tex
\section{\label{sec:Power-simulations}Power calculations}

In this section we present power calculations for three different types of 
covariance models. 
The three covariance models we consider are: (1) Exchangeable
row covariance and exchangeable column covariance; 
(2) Maximally sparse 
Kronecker structured covariance; and (3) the covariance
induced by a nonseparable stochastic blockmodel with
two row groups and two column groups.
For each covariance model, we consider the 
power as a function of parameters that control the
total correlation within a covariance matrix as well as in terms of $m$, 
the dimension of the matrix. 
In Table \ref{Flo:null_quant} we present the
95\% quantiles based on the null distributions required 
for performing level $\alpha=0.05$ tests. 


\begin{table}
\begin{center}
\begin{tabular}{rrrrrrrrr}
  \hline
Dimension $m$ & 5 & 10 & 15& 20 & 25 & 30 & 50 & 100 \\ 
  95\% quantile & 43.3 & 144.3 & 297.4 & 502.8 & 760.0 & 1064.6 & 2802.1 & 10668.4 \\ 
   \hline
\end{tabular}
\end{center}
\caption{\label{Flo:null_quant}95\% quantile of the null distribution of the test statistic 
for testing
$H_{0}$ versus $H_{1}$.
Approximation from 100,000 simulated values.}
\end{table}

\subsection{\label{sub:exchange}Exchangeable row and column covariance structure }
We first consider a submodel of the matrix normal model 
in which $\Sigma_{\row}$ and $\Sigma_{\col}$ have exchangeable 
covariance structure. In this structure, the correlation between
any two rows is a constant $\rho_{\row}$ and the correlation
between any two columns is a constant $\rho_{\col}$. 
Specifically, ${\rm cov}({\rm vec}(Y))=\Sigma_{\col}\otimes\Sigma_{\row}$ where
\begin{align*}
\Sigma_{\row}= \left(1-\rho_{{\rm r}}\right)I+\rho_{{\rm r}}{\bf 11}^{t}\ \text{and}\ 
\Sigma_{\col}=\left(1-\rho_{{\rm c}}\right)I+\rho_{{\rm c}}{\bf 11}^{t},
\end{align*}
and ${\bf 1}$ is a vector of ones of length $m$. 
We first consider a network with $m=10$ nodes 
and present the power as a function of $\rho_{\row}$ and $\rho_{\col}$
ranging from $-1/9$ to 1, where the lower bound guarantees
that the covariance matrices are positive definite.
We calculate the power on a $25\times 25$ grid in $[-1/9,1]^2$
and use a bivariate interpolation to construct the heatmap
in the top left panel of Figure \ref{flo:all}.
From the plot it is evident that the power is an increasing function
in $|\rho_{\row}|$ and $|\rho_{\col}|$. In particular, keeping $\rho_{\row}$ 
constant, the power is an increasing function of $|\rho_{\col}|$ and vice versa. 

In the top left panel of Figure \ref{flo:all} we observed that while keeping
$\rho_{\col}$ constant, the power is an increasing function of $|\rho_{\row}|$.
To study the power for higher dimensional matrices, we 
set $\rho_{\col}=0$ and vary $m$ and $\rho_{\row}$. The dashed line
in the top left panel of Figure \ref{flo:all} traces
the power function for $m=10$ and $\rho_{\col}=0$. This corresponds to the 
same style dashed line in the top right hand panel of Figure \ref{flo:all}.
The other four lines in the right hand panel represent the calculated power as a function
of $\rho_{\row}$ for different dimensions $m$, holding $\rho_{\col}=0$. 
As is expected, for each $m$, the power is an increasing function of $|\rho_{\row}|$.
Similarly, for each fixed $\rho_{\row}$ value, the power is an increasing function of $m$. 
This latter phenomenon is due to the increase in the amount of data information with the increase
in the dimension of the sociomatrix.


\begin{figure}[!ht]
\centering
\includegraphics[scale=.8]{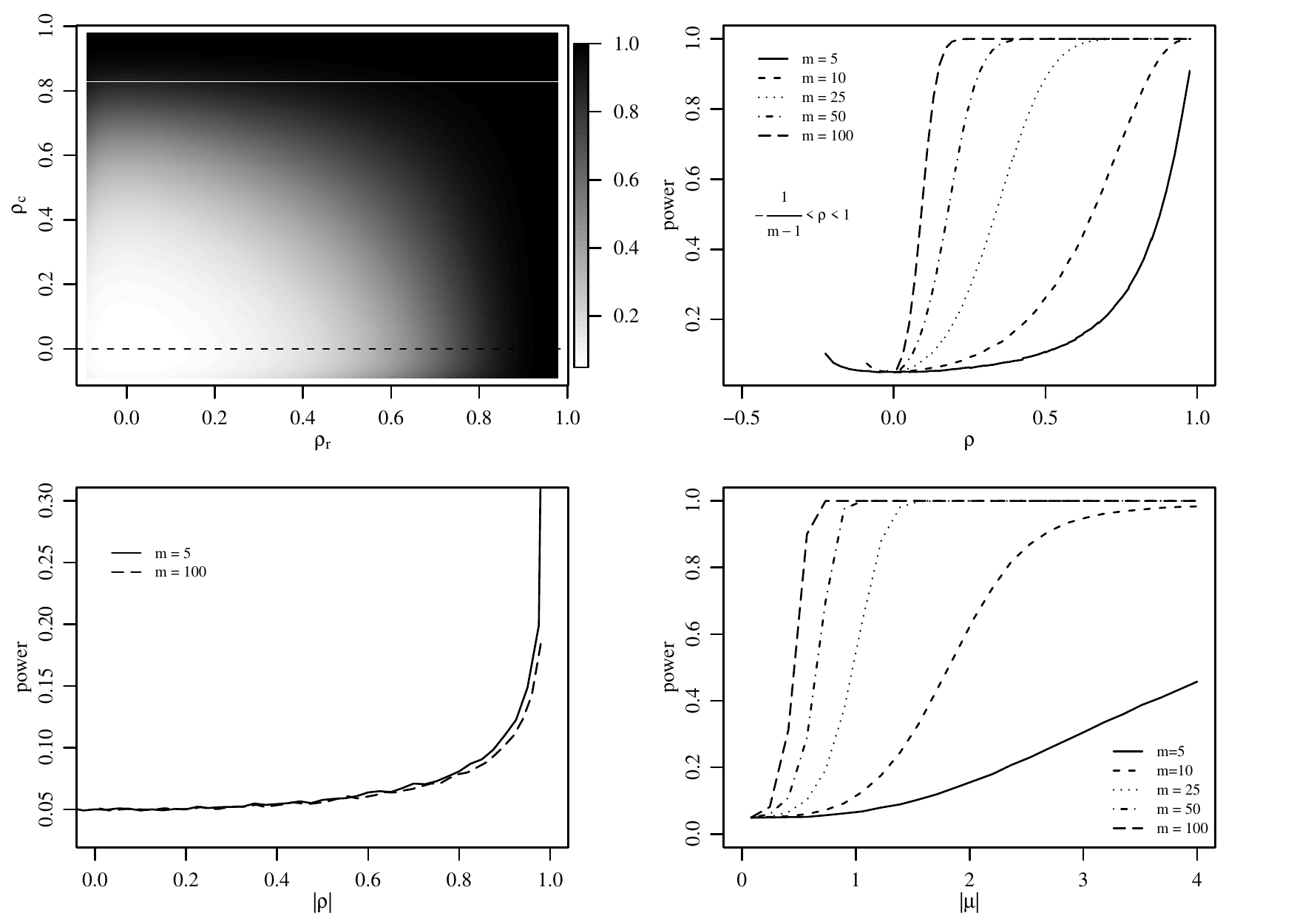}
\caption{
The top row of panels displays the power of the test under
the exchangeable covariance model of Section \ref{sub:exchange}.
The bottom left panel displays the power of the test for the
maximally sparse covariance model of Section \ref{sub:small} and
the bottom right panel displays the power of the test for the 
nonseparable stochastic blockmodel of Section \ref{sub:block}.
\label{flo:all}}
\end{figure}

\subsection{\label{sub:small}{Maximally sparse Kronecker covariance structured correlation}}

While the previous example demonstrates the power of the test in the
presence of many nonzero off-diagonal entries in the correlation matrices,
it is of interest to see if the test has any power against alternatives
that do not exhibit a large amount of correlation. For this purpose
we consider a maximally sparse Kronecker covariance structure.
We set the columns to be independent
and only the first two rows to be correlated. This can be written compactly as
 $\Sigma_{\col}=I$ and $\Sigma_{\row}=I+\rho E_{12}+\rho E_{21}$,
where $E_{ij}$ is the 0 matrix with a 1 in the $\left(i,j\right)^{\rm th}$ entry. For each
of the matrix sizes $m\in\left\{ 5,10,25,50,100\right\} $ we computed power functions
for values of $\rho$ ranging between -1 and 1.
Monte Carlo approximations to the corresponding power functions
 are presented in the bottom left panel of Figure \ref{flo:all}. We plot
the results for $m=5$ and $m=100$ and see that the power increases
monotonically as a function of $\left|\rho\right|$ for both dimensions. 
While the power of the test for a fixed $\rho$ appears to decrease
as the size of the network $m$ increases, the two curves are
nearly identical. 
We explain this as follows: while one expects that as the dimension
$m$ increases there is an increase in data information for identifying
the correlation $\rho$, the power curve is influenced more heavily
by the fact that the difference between $\Sigma_{\row}$ and the identity
matrix becomes less pronounced.
Additional power curves for a range of $m$ values 
between 5 and 100 were approximated. 
All the curves were between the $m=5$ and $m=100$
curves that are presented in the plot.
For all the power calculations, the lowest calculated power
 for values of $\rho$ close to 0 was 
always within two Monte Carlo standard errors of 0.05.

\subsection{\label{sub:block}Misspecified covariance structure}

In the Introduction we discussed a popular model
for relational data with an underlying assumption of stochastically
equivalent nodes called the stochastic blockmodel. 
Straightforward calculations show
that in general the covariance
 induced by a stochastic blockmodel is 
nonseparable, but still induces correlations
among the rows and among the columns. As such, 
we are interested in evaluating the power
of our test against such nonseparable
alternatives.

A stochastic blockmodel can be represented in terms of 
multiplicative latent variables. Specifically, we can write the relationship
$y_{ij}=u_i^t W v_j+\epsilon_{ij}$ where $u_i$ and $v_j$ are 
latent vectors representing the row group membership of node $i$ 
and the column group membership of node $j$. $W$ is a matrix of means 
for the different group memberships and $\epsilon_{ij}$ is iid random noise.
For the purposes of this power calculation we consider a simple
setup where each node belongs to one of two row groups and 
one of two column groups with equal probability. 
We let 
$W=\bigl(\begin{smallmatrix} 0&-\mu \\\mu&0\end{smallmatrix} \bigr)$
depend on a single parameter $\mu>0$. Under this choice
of $W$, we have $E[Y]=0$ and 
$E[{\bf 1}^t Y {\bf 1}]=0$. 
Since there are only two groups, the latent group membership
vectors can be written as $u_i=(u_{i1},1-u_{i1})$ and 
$v_j=(v_{j1},1-v_{j1})$ where $u_{i1}$ and $v_{j1}$ are independent 
Bernoulli($1/2$) random variables.

The bottom right panel of Figure \ref{flo:all} presents the power calculations
for dimensions $m\in\{5,10,25,50,100\}$ and $|\mu|\in[0,4]$. The power 
of the level $\alpha=0.05$ test is increasing in $|\mu|$ which is a desirable property
for this blockmodel since
as $|\mu|$ grows
the difference in the means for the groups becomes greater. 
The power of the test also increases with the 
dimension $m$.

%% file: extend.tex
\section{\label{sec:Extensions-and-data}Extensions and Applications}

In this section we develop several extensions of the proposed test, 
and illustrate their use in the context of two data analysis examples.
In the first example, we show how the 
test can be extended to accommodate
a missing diagonal, an unknown non-zero mean,
and heteroscedastic replications. 
The second example illustrates the use of the test for 
binary network data, a common type of relational data.

%

\subsection{Extensions and continuous data example}

International trade data on the value of exports from
country to country is collected by the UN on a yearly basis and disseminated
through the UN Comtrade website: \url{http://comtrade.un.org}.
In this section we consider measures of total exports between
 twenty-six mostly large,
well developed countries with high gross domestic product collected
from 1996 to 2009 (measured in 2009 dollars). Specifically,
we are interested in evaluating evidence for correlations
among exporters and among importers.
As trade between countries
is relatively stable across years, we analyze the yearly change in
log trade values, resulting in thirteen measurements 
(for the fourteen years of data)
for every country-country
pair. The data takes the form of a three way array 
$Y=\left\{ Y_{ijk}:i,j\in\{1,\dots,26\},k\in\{1,\dots,13\}\right\} $
where $i$ and $j$ index the exporting and
importing countries respectively and $k$
indexes the year. Exports from a country to 
itself are not defined, and so entries $Y_{iik}$ are ``missing''.

We consider a model for trade of the form,
\begin{align}
Y_{ijk}= \beta_1 x_{ik} + \beta_2 x_{jk}+ \epsilon_{ijk},\label{eq:gravity2}
\end{align}
where $x_{ik}$ is the difference in log gross
domestic product of country $i$ between years
$k$ and $k-1$ (theoretical development of this model is
available in the economics literature, see \cite{tinbergen1962shaping}, and \cite{bergstrand1985gravity,bergstrand1989generalized}). 
We use GDP data 
collected by the
World Bank through \url{http://data.worldbank.org/}
to obtain OLS estimates of $\beta_1$ and $\beta_2$.
To investigate the correlations among importers and among exporters
we collect the residuals
$e_{ijk}=Y_{ijk}-\hat{Y}_{ijk}$ into thirteen matrices, $E_{\cdot\cdot k}$
for $k=1,\dots,13$.
Figure \ref{Flo:data_eigen} plots the first two eigenvectors
of moment estimates of pairwise row and column correlation matrices based
on $E_{\cdot\cdot1},\dots,E_{\cdot\cdot\left(13\right)}$.
We observe systematic geographic patterns in both panels of the figure 
suggesting evidence that the $\epsilon_{ijk}$ are not independent.
To evaluate this evidence formally by testing 
 for dependence of the $\epsilon_{ijk}$ we
  extend the conditions under which the test developed in this article
  is applicable.
Specifically, we must accommodate the following features of
this data:
 a missing diagonal ($Y_{iik}$ is not defined
for all $i$ and $k$), multiple observations (13 data points),
and a nonzero mean structure (of the form \eqref{eq:gravity2}).
\begin{figure}
\begin{center}\includegraphics[scale=.9]{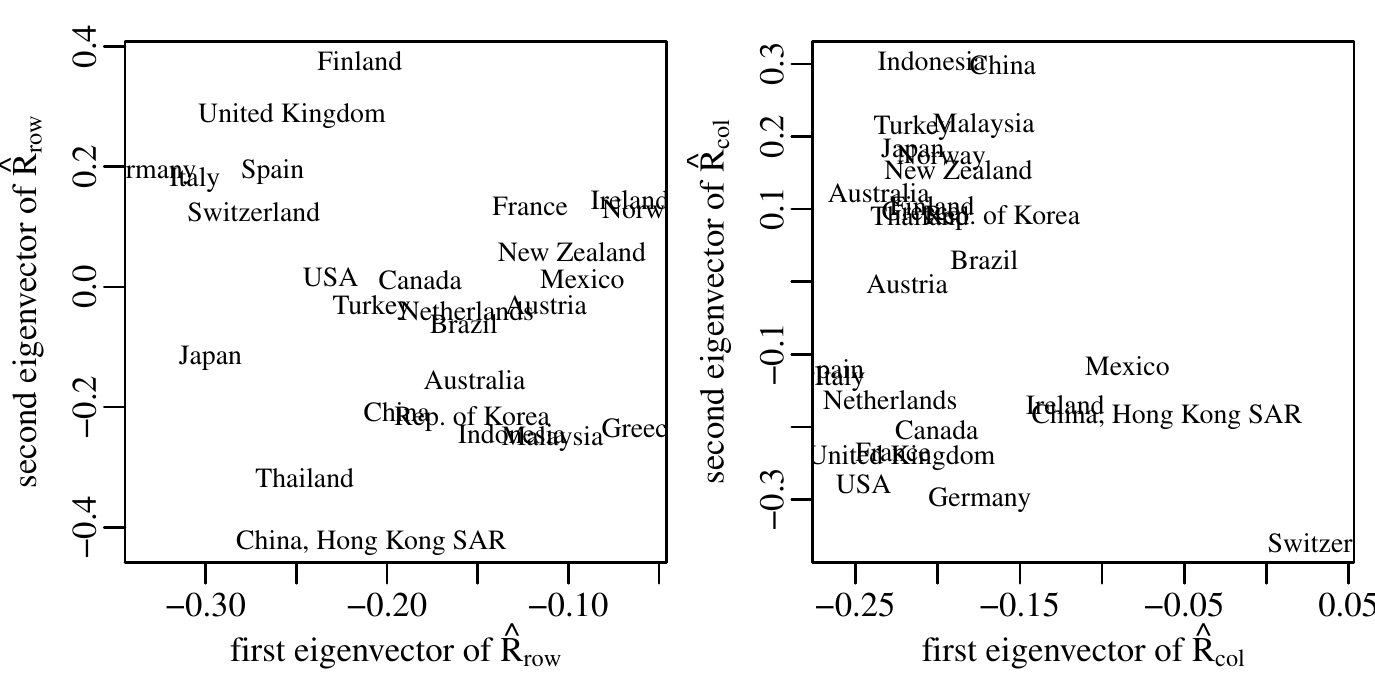}\end{center}

\caption{\label{Flo:data_eigen}Plots of the first two eigenvectors of the
estimates $\hat{R}_{{\rm row}}$ and $\hat{R}_{{\rm col}}$ of the
row and column correlation matrices. Proximity of countries in the eigenspace
indicates a positive correlation.}

\end{figure}

\paragraph{\label{sub:Missing-diagonal}Missing diagonal:}

In relational datasets, 
the relationship of an actor to himself is typically undefined,
meaning that 
the relational matrix 
$Y$ has an undefined diagonal. 
It is common to
 treat the entries of an undefined diagonal as missing at random
and to use a data augmentation procedure to recover a complete data
matrix, applying the analysis to the complete data. In the context
of this article,  this approach would allow us to treat the whole
data matrix as a draw from a matrix normal distribution and perform
our test exactly as outlined in Section \ref{sec:Likelihood-ratio-test}.
In this section we describe an augmentation procedure
that does not require distributional assumptions for the diagonal
elements. The procedure
produces a data matrix $\tilde{Y}$ that we use to 
calculate the test statistic $T(\tilde{Y})$. 
Specifically, $\tilde{Y}$ replaces the undefined diagonal of $Y$ with zeros,
keeping the rest of the data matrix the same. 
We show that $T(\tilde{Y})$ is invariant under
diagonal transformations and thus we can approximate the
null distribution of the test statistic based on for data drawn from
a matrix normal distribution where the diagonal entries are replaced
with zeros. 

Consider square matrices $Y$ and $\tilde{Y}$ where $Y\sim N_{m\times m}\left(0,D_{\row},D_{\col}\right)$
while $\tilde{Y}$ is distributed identically to $Y$ 
except the diagonal entries are replaced with zeros.
Similarly define square matrices $Y_{0}\sim N_{m\times m}\left(0,I,I\right)$
and $\tilde{Y}_{0}$. We showed in Section \ref{sec:Likelihood-ratio-test} that
the distribution
of the likelihood ratio test statistic is invariant under transformations
by diagonal matrices on the left and right, 
that is $T(Y)\overset{{\rm d}}{=}T(Y_{0})$. 
We now show that $T(\tilde{Y})\overset{{\rm d}}{=}T(\tilde{Y}_{0})$.
It is immediate that since $Y\overset{{\rm d}}{=}D_{\row}^{1/2}Y_{0}D_{\col}^{1/2}$
we have $\tilde{Y}\overset{{\rm d}}{=}D_{\row}^{1/2}\tilde{Y}_{0}D_{\col}^{1/2}$,
as zeros on the diagonal are preserved by left and right diagonal
transformations and the off diagonal entries $\left(i,j\right)$ are
normally distributed with variance $D_{\row,i}D_{\col,j}$.
As in Section \ref{sec:lrts}, we appeal to the equivariance of a unique MLE to show that 
$T(\tilde{Y})=T(\tilde{Y}_{0})$. 
The argument is identical to the 
one appearing in the paragraph following Equation \eqref{eq:test_clean} 
on page \pageref{eq:test_clean}
and so we do not reproduce it here.
Since $T(\tilde{Y})\overset{{\rm d}}{=}T(\tilde{Y}_{0})$,
we can approximate the null distribution and calculate the relevant quantiles
for the test statistic with a simple update to the algorithm at the end of Section \ref{sec:lrts}: 
\begin{enumerate}
\item Simulate  $\tilde Y_{0}^{1},\dots,\tilde Y_{0}^{S}\overset{\rm iid}{\sim}\mathcal{L}\left(\tilde{Y}_{0}\right)$, where $\mathcal{L}(\tilde{Y}_{0})$ denotes the distribution
of $\tilde{Y}_{0}$;
\item Let $\hat{T}_q=\min\{T(\tilde Y_0^q):\sum_{s=1}^S 1[T(\tilde Y_0^q)\geq T(\tilde Y_0^s)]/S\geq q\}$.
\end{enumerate}

\paragraph{\label{sub:3-way-array}Repeated observations:}

The test we discussed in this article is designed for a single observation.
However, the test conveniently generalizes to the situation
in which multiple observations are available.
We will consider two types of additional
observations: independent homoscedastic observations and independent
heteroscedastic observations. First, if there are $p$ independent
identically distributed observations, we note that likelihood equations
of Section 2 can be rewritten as \begin{eqnarray*}
mp\hat{D}_{\row}=\sum Y_{i}\hat{D}_{\col}^{-1}Y_{i}^{t}\circ I &  & mp\hat{D}_{\col}=\sum Y_{i}^{t}\hat{D}_{\row}^{-1}Y_{i}\circ I\\
mp\hat{\Sigma}_{\row}=\sum Y_{i}\hat{\Sigma}_{\col}^{-1}Y_{i}^{t} &  & mp\hat{\Sigma}_{\col}=\sum Y_{i}^{t}\hat{\Sigma}_{\col}^{-1}Y_{i}.\end{eqnarray*}
The likelihood remains bounded and the form of the test statistic
is identical to Equation \eqref{eq:test_clean}. When the observations are heteroscedastic 
the likelihood equations
(included in the proof of Theorem \ref{thm:3wayext} in the Appendix)
are more complicated because of the need to estimate the variability
along the replications (we refer the reader to \citet{hoff2011separable}
for an exposition on the general class of array normal distributions
and estimation procedures). 
\begin{thm}
\label{thm:3wayext} Let $Y_{1},\dots,Y_{p}$ be independent random
matrices distributed as $Y_{i}\sim N_{m\times m}\left(0,d_{i}\Sigma_{\row},\Sigma_{\col}\right)$.
Then $\forall p\geq1$ 
the likelihood is bounded as a function of
the covariance matrices $\Sigma_{\row}$
and $\Sigma_{\col}$ and the variance parameters $d_{1},\dots,d_{p}$. 
\end{thm}
A proof is in the Appendix. Theorem~\ref{thm:3wayext} extends the literature
on maximum likelihood estimation for proportional covariance models 
from natural exponential families to the matrix normal family which is
a curved exponential family
 \citep{eriksen1987proportionality,flury1986proportionality,jensen1987estimation,jensen2004estimation}.
Due to Theorem \ref{thm:3wayext}, we
can modify the test statistic to test 
 $H_{0}^{\prime}:D_{{\rm obs}}\in\mathcal{D}_{+}^{p},\ D_{\col},D_{\row}\in\mathcal{D}_{+}^{m}$
vs $H_{1}^{\prime}:D_{{\rm obs}}\in\mathcal{D}_{+}^{p},\ \Sigma_{\col},\Sigma_{\row}\in\mathcal{S}_{+}^{m}$:\begin{align*}
T\left(Y_{1},\dots,Y_{p}\right) & = l(\hat{D}_{\rm obs}^{\rm null},\hat{D}_{\row},\hat{D}_{\col};Y)-l(\hat{D}_{\rm obs}^{\rm alt},\hat{\Sigma}_{\row},\hat{\Sigma}_{\col};Y)\\
&= \log\left|\hat{D}_{{\rm obs}}^{{\rm null}}\otimes\hat{D}_{\col}\otimes\hat{D}_{\row}\right|-\log\left|\hat{D}_{{\rm obs}}^{{\rm alt}}\otimes\hat{\Sigma}_{\col}\otimes\hat{\Sigma}_{\row}\right|.\end{align*}
We can again approximate the null distribution of the test statistic
due to the invariance of the test statistic $T\left(Y_{1},\dots,Y_{p}\right)$
under diagonal transformations of the data along all three modes.
The results on missing diagonal elements (above) and a non-zero mean (below)
are also immediately applicable to $T\left(Y_{1},\dots,Y_{p}\right)$
above.

\paragraph{Relaxing the mean zero assumption: \label{sub:Relaxing-the-mean}}

The reference distributions for the test statistics developed in Section
\ref{sec:Likelihood-ratio-test} are based on the assumption that
${\rm E}[Y]=0$, a strong assumption that is unlikely
to be true for any observed dataset. While treating the mean as a
nuisance parameter is tempting, the likelihood function under the
alternative model is unbounded when estimating a mean matrix and two
covariance matrices simultaneously. We propose to first fit
a regression based mean to the data assuming the entries in the data matrix are independently
distributed with the same variance parameter and to then perform the test based on
the demeaned data. We consider the following regression framework for the mean:
$y_{ij} =  \beta^t x_{ij}+\epsilon_{ij}$.
The regressor $x_{ij}$ is a $p$-dimensional vector that can
include features of node $i$, features of node $j$ and dyadic
features for nodes $i$ and $j$. The $\epsilon_{ij}$ are assumed to
be independent and identically distributed errors. Writing this in
vector notation as ${\rm vec}\left(Y\right)=X\beta+{\rm vec}\left(\epsilon\right)$
where $X=\left(x_{12}^{t}\cdots x_{\left(m-1\right)m}^{t}\right)^{t}$ and $\epsilon$
is an $m\times m$ matrix,
the OLS estimate of $\beta$ is  
$\hat{\beta}  = \left(X^{t}X\right)^{-1}X^{t}{\rm vec}\left(Y\right)$.
Under mild regularity conditions on the distribution of the explanatory
variables $X$ and the row and column variances for new nodes, the
OLS estimate $\hat{\beta}$ is a consistent estimate of $\beta$.
This motivates us to base the test statistic on 
${\rm vec}\left(\hat{\epsilon}\right)={\rm vec}\left(Y\right)-X\hat{\beta}$,
the residuals of the regression, as we expect the distribution of ${\rm vec}\left(\hat{\epsilon}\right)$ 
to be close to the distribution of ${\rm vec}\left(\epsilon\right)$ for large $m$.
The null distribution of the test statistic $T$ based on the residuals
from the regression is not identical to the one derived in Section
\ref{sec:Likelihood-ratio-test} and no explicit computation of the
new null distribution is readily available. However, we have observed
via simulation that the level of the test based on the estimated residuals
$\hat{\epsilon}$ appears to be asymptotically correct and that the test statistic
based on $\epsilon$ and $\hat{\epsilon}$ appear to have the same limiting distributions.

%% file: data_format.tex
\paragraph{\label{sec:Trade-data-example}Application to international trade
data:}

Figure \ref{Flo:data_eigen} suggested that there is evidence of residual
dependence among exporters and among importers 
based on
additional information about the data (the relative geographic positions of the countries). 
Above we developed the tools to 
test for independence of $\epsilon_{ijk}$ in \eqref{eq:gravity2}
using the likelihood ratio test proposed in Section \ref{sec:Likelihood-ratio-test}.
Formally, we are testing the
null hypothesis $H_{0}^{\prime}:D_{{\rm time}}\in\mathcal{D}_{+}^{13},\ D_{\rm imp},D_{\rm exp}\in\mathcal{D}_{+}^{26}$
versus the alternative hypothesis $H_{1}^{\prime}:D_{{\rm time}}\in\mathcal{D}_{+}^{13},\ \Sigma_{\rm imp},\Sigma_{\rm exp}\in\mathcal{S}_{+}^{26}$. 
The approximate 95\% quantile of the distribution of the
test statistic when the data are missing diagonal entries
under the null is $729.8$.
Setting $E_{iik}=0$ for all $i$ and $k$,
the test statistic for the data is $T(E_{\cdot\cdot1},\dots,E_{\cdot\cdot13})=3354$.
This value is much greater than the 95\% quantile confirming
that we should reject the independence of the $\epsilon_{ijk}$. 
It is thus inappropriate
to assume that the exporters and importers are independent.

\subsection{\label{sub:data_pvalues}Application to binary protein-protein interaction network}

So far we have developed a testing procedure for 
the presence of row and column
correlations in relational matrices
within the framework of matrix normal and general matrix variate
elliptically contoured distributions. 
In this section we propose a methodology
that allows us to evaluate the presence of row and column
correlations for binary relational data,
where the observed
network is represented by a sociomatrix matrix $A$ where
$a_{ij}$ describes the relationship
from node $i$ to node $j$. When the entries of $a_{ij}$ are binary 
indicators of a relationship from $i$ to $j$, 
the matrix $A$ can be viewed
as the adjacency matrix of a directed graph.  
In this example we use the protein-protein interaction 
data of \cite{butland2005interaction},
which consists of a record of 
interactions between 
$m=270$ essential proteins of {\emph{E. coli}}.
The data are organized into a $270\times270$ binary matrix $A$, where
$a_{ij}=1$ if protein $j$ binds to protein $i$ and $a_{ij}=0$ otherwise. The network
has one large connected component with 234 nodes as shown in the left hand side of 
Figure \ref{Flo:post_p_val}.
In this case diagonal elements of the matrix $A$ are meaningful
 since proteins may bind to themselves.

A popular class of models for the analysis of 
such data is based on representing the relations $a_{ij}$
as functions of latent normal random variables \citep{hoff_2005,hoff_2008_nips}.
For the protein-protein interaction data
 we propose to use an asymmetric version of the eigenmodel
of \cite{hoff_2008_nips},  a type of reduced rank
 latent variable model where the relationship
between nodes $i$ and $j$ is characterized by multiplicative latent sender and receiver
effects. The model can be written as:
\begin{align*}
a_{ij}&=1[y_{ij}>\gamma], &
y_{ij}&=u_i^t  v_j + \epsilon_{ij}, &
Y&=UV^t +E,
\end{align*}
where $\epsilon_{ij}\overset{\rm iid}{\sim}N(0,1)$ and $u_i,v_j\in\mathbb{R}^R$ for $R<270$.
Considering $Y$ as a matrix variate normal
variable it is immediate that $E[YY^t]=U(V^tV)U^t+I$ and
$E[Y^tY]=V(U^tU)V^t+I$ and so the heterogeneity in 
$U$ describes the row covariance $\Sigma_{\row}$
while the heterogeneity in $V$ describes the column covariance
$\Sigma_{\col}$.

We propose using the test developed in this article to evaluate
how well models of rank $R$ capture the dependence in the data.
Specifically,
we fit the above model for multiple values of $R$,
and for each value we approximate the posterior 
distribution of the test statistic. If the rank-$R$ model
is sufficient
for capturing the row and column correlations found in 
the data, we do not expect to have evidence
to reject the null of independence.
\cite{hoff_2008_nips} outlines
a Markov chain Monte Carlo algorithm for fitting the above model.
Following the procedure described in \cite{thompson2007fuzzy} we 
apply the 
testing procedure to 
draws from the posterior distribution 
of $Y-UV^t$
constructed via MCMC and for each
test statistic we calculate a $p$-value. These $p$-values are termed ``fuzzy $p$-values''
as their distribution provides a description of the uncertainty about the 
$p$-value that results from not observing $Y$, $U$ and $V$. 

\begin{figure}[ht]
\centering
\includegraphics[scale=.65]{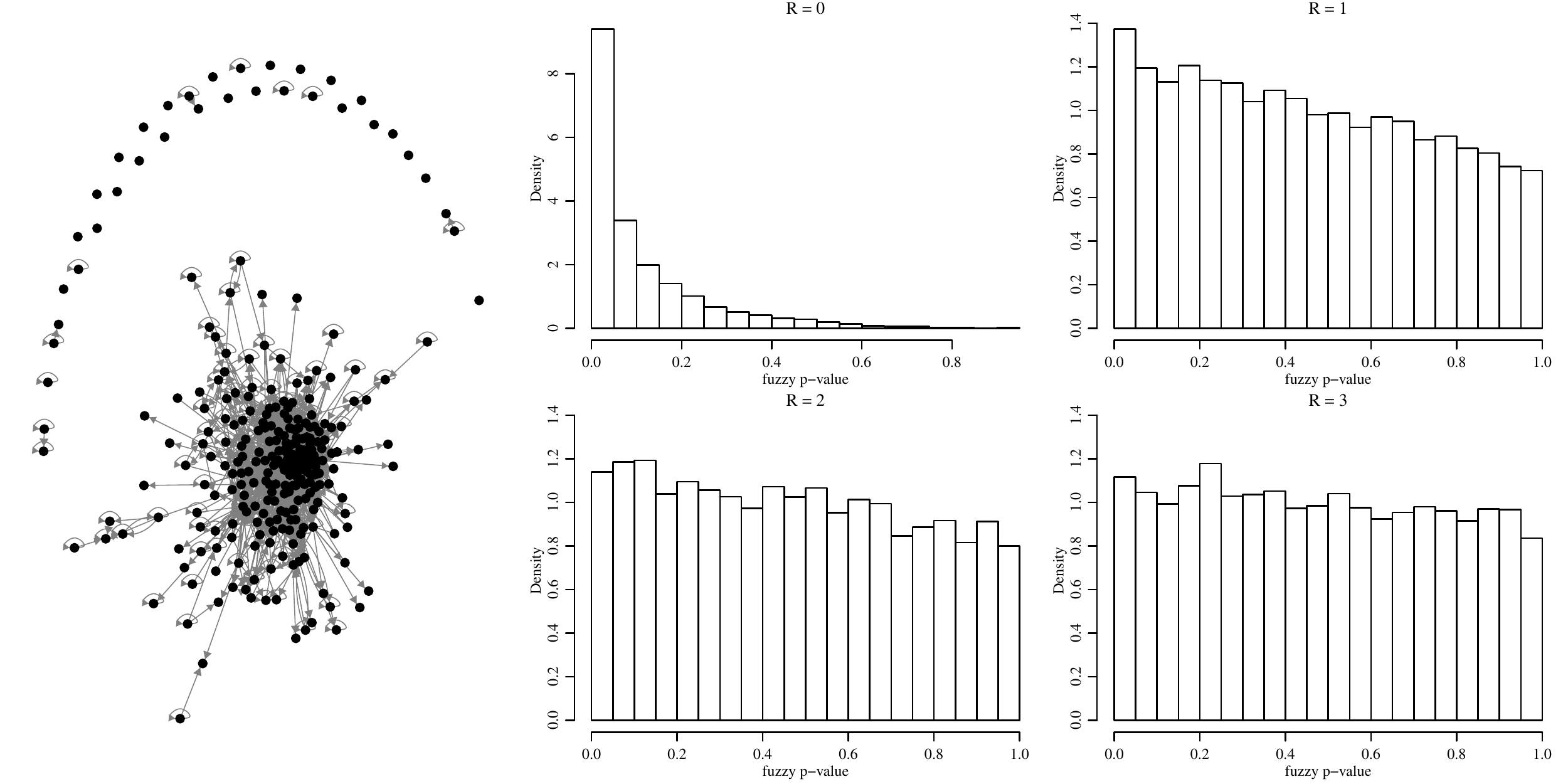}
\caption{Protein-protein interaction data and histograms of fuzzy $p$-values for models of ranks $R\in\{0,1,2,3\}$. The bins have a width of $0.05$.
}\label{Flo:post_p_val}
\end{figure}

As this is a very sparse network (the interaction rate is $\bar{A}=0.03$),
we expect a low rank approximation to be appropriate. In fact, analysis 
using cross validation of a symmetrized
version of this data identified $R=3$
to be an appropriate rank in \cite{hoff_2008_nips}. In the right hand
side of Figure \ref{Flo:post_p_val} we present the distributions of the fuzzy $p$-values 
for $R\in\{0,1,2,3\}$. 
A visual inspection of the the fuzzy $p$-values in the
four panels of the figure provides
evidence  
about the rank of the latent factors. 
For example, under the $R=0$ model, the $y_{ij}$s are independent 
and identically distributed, and so the graph represented by the 
adjacency matrix $A$ is a simple random graph. The fuzzy $p$-values 
are concentrated
at a value lower than 0.05 
suggesting a high probability of rejecting the null if $Y$ were
observed.
For $R\in\{1,2\}$ the 
fuzzy $p$-values are no longer concentrated lower than 0.05, 
but the distribution is
skewed to the right, which we take as evidence that there is 
correlation in $Y$ that is not captured
by the rank 1 and rank 2 models. 
The fuzzy $p$-values provide little evidence of 
residual dependence in $Y$ for models
of rank $R\geq 3$.

%% file: discuss.tex
\section{Discussion}

In this article we presented a likelihood ratio test for relational
datasets. Unlike the previous testing literature 
for matrix normal models
that required multiple
observations, and concentrated on testing a null of separable covariances
versus an unstructured alternative, we proposed testing a null of
no row or column correlations versus an alternative of full row and
column correlations using a single observation of a network. While
the form of the null distribution of the test statistic is intractable,
we are able to simulate a reference distribution for the test statistic under 
the null due to its invariance to left and right diagonal transformations of the data.
In the power simulations of Section \ref{sec:Power-simulations} we
demonstrated the power of the test against maximally
sparse and nonseparable alternatives.

This test can be applied to a wide variety of relational data. While
the test was developed using the matrix-normal model,
we have shown that this
distributional assumptions can be greatly relaxed.
Specifically, if we consider a data
matrix $Y$ with an arbitrary matrix variate elliptically contoured
 distribution that is centered at the
zero matrix, the test statistic for testing for correlation among
the rows and among the columns of $Y$ is identical
to that of the matrix normal case. We have also
demonstrated that the test can accommodate frequently observed
 features of relational data such as non-zero mean, missing 
 diagonal and multiple observations.  
In Section \ref{sub:data_pvalues} we demonstrated an
application of the theory developed in this paper to 
binary network data 
where the matrix $Y$ is an adjacency matrix.
The method we describe for binary data can be
extended to ordinal and discrete data
that can be modeled via a latent 
matrix variate elliptically contoured distribution.

Once we reject the null hypothesis of independence
among the rows and among the columns of a relational
matrix, we are faced with the challenge of modeling
the dependence in the data.
As shown in Section \ref{sub:MLE}, for a mean zero
matrix normal distribution, the MLE is not unique.
Specifically, there is an MLE for each 
$\Sigma_{\row}\in\mathcal{S}^{m}_{+}$ given by
$( \Sigma_{\row} ,   Y^{t}{\Sigma}_{\row}^{-1}Y/m)$.
We have observed 
in separate work
that it is possible to distinguish 
between MLEs by considering their risk. However,
other than in very specialized cases (such as
equal eigenvalues of $\Sigma_{\row}$ and $\Sigma_{\col}$),
obtaining 
analytic results for identifying risk optimal MLEs is difficult. 
The presence of a non-zero mean leads to an unbounded
likelihood and further complicates the problem of estimation.
Several authors have recently considered
Bayesian and penalized likelihood approaches
 this estimation problem.
\cite{bonilla2008multi} and \cite{yu2007stochastic} studied hierarchical
Gaussian Process priors in the context of a classification problem.
In our context, this approach results in a matrix normal
prior for the mean parameters and inverse Wishart priors
for the row and column covariance matrices. A second approach
based on a mixture of independent  $L_1$ and $L_2$ penalties
on the row and column precision matrices was proposed by 
\cite{allen2010transposable}.


Computer code and data for the results in Sections 3 and 4 are available
at the authors' websites.

%% file: append.tex
\appendix

\section{Proofs}

\proof[Proof of Theorem \ref{lem:null_unique}]To show that the solutions to the likelihood
equations provide a unique minimizer to the scaled log likelihood
function (Equation \ref{eq:ld}) we will show that the Hessian of
$l$ evaluated at the solutions is strictly positive definite and
then demonstrate that only a single solution is possible. We rewrite
Equation \ref{eq:ld} here, explicitly stating that we will be considering
diagonal matrices \begin{eqnarray*}
l\left(D_{\row},D_{\col};Y\right)=-2\log L\left(D_{\row},D_{col};Y\right) & = & {\rm tr}\left[D_{\row}^{-1}YD_{\col}^{-1}Y^{t}\right]-\log\left|D_{\col}^{-1}\otimes D_{\row}^{-1}\right|+c,\end{eqnarray*}
 writing for simplicity $\Psi=D_{\row}^{-1}$ and $\Gamma=D_{\col}^{-1}$
we take first derivatives with respect to the diagonal matrices of
$\Gamma$ and $\Psi$: \begin{eqnarray}
\frac{\partial l\left(D_{\row},D_{\col};Y\right)}{\partial_{\Psi}} & = & Y\Gamma Y^{t}\circ I-m\Psi^{-1}\label{eq:partial_p_mat}\\
\frac{\partial l\left(D_{\row},D_{\col};Y\right)}{\partial_{\Gamma}} & = & Y^{t}\Psi Y\circ I-m\Gamma^{-1},\label{eq:partial_g_mat}\end{eqnarray}
yielding the familiar equations used to find the maximizers of the
likelihood. Considering the singular value decomposition of $Y=ALB^{t}$,
the above can also be written as partial derivatives with respect
to the entries of $\Psi$ and $\Gamma$ (since these are diagonal matrices,
the index $k$ refers to the $k^{\rm th}$ row, $k^{\rm th}$ column  entry in the matrix): 
\begin{eqnarray}
\frac{\partial l\left(D_{\row},D_{\col};Y\right)}{\partial_{\Psi_{j}}}
 & = & \sum_{ikm}L_{i}L_{m}\Gamma_{k}\left(A_{jm}A_{ji}B_{km}B_{ki}\right)-\frac{m}{\Psi_{j}}\label{eq:partial_p}\\
\frac{\partial l\left(D_{\row},D_{\col};Y\right)}{\partial_{\Gamma_{k}}} & = & \sum_{ijm}L_{i}L_{m}\Psi_{j}\left(A_{jm}A_{ji}B_{km}B_{ki}\right)-\frac{m}{\Gamma_{k}}\label{eq:partial_g}\end{eqnarray}
To compute the Hessian, we take derivatives of equations \ref{eq:partial_p}
and \ref{eq:partial_g}, yielding the second partial derivatives of
$l_{{\rm D}}$: \begin{eqnarray*}
\frac{\partial l\left(D_{\row},D_{\col};Y\right)}{\partial_{\Psi_{j}\Gamma_{k}}} & = & f\left(j,k\right)=\sum_{im}L_{i}L_{m}\left(A_{jm}A_{ji}B_{km}B_{ki}\right)
  =  Y_{jk}^{2}\\
\frac{\partial l\left(D_{\row},D_{\col};Y\right)}{\partial_{\Psi_{j}\Psi_{l}}}=\frac{\partial l\left(D_{\row},D_{\col};Y\right)}{\partial_{\Gamma_{k}\Gamma_{m}}} & = & 0\\
\frac{\partial l\left(D_{\row},D_{\col};Y\right)}{\partial_{\Psi_{j}\Psi_{j}}} & = & \frac{m}{\Psi_{j}^{2}}\\
\frac{\partial l\left(D_{\row},D_{\col};Y\right)}{\partial_{\Gamma_{k}\Gamma_{k}}} & = & \frac{m}{\Gamma_{k}^{2}}.\end{eqnarray*}
As such, we can write the Hessian matrix $H$ as \begin{eqnarray*}
H & = & \left[\begin{matrix}m\Psi^{-2} & F\\
F^{t} & m\Gamma^{-2}\end{matrix}\right]\end{eqnarray*}
where $F=\left[f\left(j,k\right)\right]_{j,k}$. Our first observation
is that $F$ is an everywhere positive matrix since $f\left(j,k\right)=Y_{jk}^{2}>0\forall j,k$
(since $P\left(Y\not=0\right)=1$).

To show that $l_{{\rm F}}$ is minimized at the solutions to the likelihood
equations \ref{eq:partial_p_mat} and \ref{eq:partial_g_mat} we will
show that the Hessian $H$ is strictly positive definite at the solutions.
For that, we will verify Sylvester's criterion: a matrix $H$ is positive
definite if and only if all of its leading minors are positive (or
equivalently its trailing minors). First we note that the Kronecker
product of the covariances leads to a nonidentifiability in the scale
of the individual matrices, and so WLOG we let $\Psi_{1}=1$. We consider
the reparametrized problem and its' Hessian $\tilde{H}=H_{-1,-1}$,
the Hessian of the original problem with the first row and column
removed. Now, the boundedness of the likelihood function implies that
the $m-1$ leading minors are positive (since $\hat{\Psi}$ is a positive
diagonal matrix) and so we are left with verifying that the remaining
$m$ minors are positive. Abusing notation a bit and writing $\Psi$
to correspond to the reparametrized version of row precisions, we
get the first minor that includes entries other than those in $\Psi$
is \begin{eqnarray*}
\left|\begin{matrix}m\hat{\Psi}^{-2} & F_{\cdot,1}\\
F_{1,\cdot} & m\hat{\Gamma}_{1}^{-2}\end{matrix}\right| & = & \left|\frac{m}{\hat{\Gamma}_{1}^{2}}-F_{\cdot,1}\frac{\hat{\Psi}^{2}}{m}F_{1,\cdot}\right|\left|m\hat{\Psi}^{-2}\right|\\
 & := & \left|a\right|\left|b\right|\end{eqnarray*}
Clearly, $\left|b\right|>0$ as it is simply the previous minor. Now,
$a$ does not satisfy the first derivative of $l\left(D_{\row},D_{\col};Y\right)$
with respect to $\Gamma_{1}$ (Equation \ref{eq:partial_g}) and more so we note
that \begin{align*}
\frac{m}{\hat{\Gamma}_{1}^{2}} = \frac{1}{m}\left(\sum_{j}\hat{\Psi}_{j}F_{j,1}\right)^{2}
  =  \frac{1}{m}\left(F_{\cdot,1}\hat{\Psi}^{2}F_{1,\cdot}+c\right)\end{align*}
where $c$ is always positive since it is a sum of positive values.
Thus, $\left|a\right|=c/m>0$. We can apply this approach to the remaining
$m-1$ minors, where the $k$th minor is given by 
$M_k = |a|M_{k-1}$
where $\left|a\right|>0$. Thus we have verified Sylvester's criterion
and have demonstrated that the Hessian of $l\left(D_{\row},D_{\col};Y\right)$
is strictly positive definite at the solution to the likelihood equations
which means we have a local minimum of the scaled log likelihood function
(or a local maximum of the likelihood function). 

To show uniqueness we apply the Mountain Pass Theorem \citep[page 223 of][]{courant2005dirichlet}:
Since $l\left(D_{\row},D_{\col};Y\right)$ is smooth, differentiable
and coercive (that is $l\left(D_{\row},D_{\col};Y\right)\rightarrow\infty$
as $\left|D_{\col}\otimes D_{\row}\right|\rightarrow\infty$),
we have that if there are two critical points $x_{1}$ and $x_{2}$
that are strict minima (strictly positive definite Hessian), then
there must be another critical point, distinct from $x_{1}$ and $x_{2}$,
that is \emph{not} a relative minimizer of $l$. This contradicts
the above notion that the Hessian is strictly positive definite for
every critical point. \qed


\proof[Proof of Theorem \ref{thm:3wayext}]For $p$ random variables $Y_{1},\dots,Y_{p}$
where $Y_{i}\sim N_{m\times m}\left(0,d_{i}\Sigma_{\row},\Sigma_{\col}\right)$,
we write $D=\left(d_{1},\dots,d_{p}\right)$ and the scaled log likelihood
function as \begin{eqnarray*}
l\left(D,\Sigma_{\row},\Sigma_{\col}|Y_{1},\dots,Y_{p}\right) & \propto & \sum_{i=1}^{p}\frac{1}{d_{i}}{\rm tr}\left(Y_{i}\Sigma_{\col}^{-1}Y_{i}^{t}\Sigma_{\row}^{-1}\right)-mp\log\left|\Sigma_{\col}^{-1}\right|-mp\log\left|\Sigma_{\col}^{-1}\right|-m^{2}\log\left|D^{-1}\right|.\end{eqnarray*}
Taking first derivatives with respect to $\Sigma_{\row}^{-1}$, $\Sigma_{\col}^{-1}$
and $d_{i}^{-1}$ and setting them equal to zero yields the likelihood
equations: \begin{eqnarray*}
mp\Sigma_{\row} & = & \sum\frac{1}{d_{i}}Y_{i}\Sigma_{\col}^{-1}Y_{i}^{t}\\
mp\Sigma_{\col} & = & \sum\frac{1}{d_{i}}Y_{i}^{t}\Sigma_{\row}^{-1}Y_{i}\\
m^{2}d_{i} & = & {\rm tr}\left(Y_{i}\Sigma_{\col}^{-1}Y_{i}^{t}\Sigma_{\row}^{-1}\right).\end{eqnarray*}
We note that when holding $\Sigma_{\row}$ and $\Sigma_{\col}$ constant,
$l\left(D,\Sigma_{\row},\Sigma_{\col}\right)$ is a strictly convex function
of $D^{-1}$ and so with probability 1 it attains a global minimum
at points that satisfy the first derivative condition $m^{2}d_{i}={\rm tr}\left(Y_{i}\Sigma_{\col}^{-1}Y_{i}^{t}\Sigma_{\row}^{-1}\right)$.
We define the profile likelihood \begin{eqnarray*}
g\left(\Sigma_{\row}^{-1},\Sigma_{\col}^{-1}\right) & = & \inf_{D^{-1}\in\mathbb{R}_{+}^{p}}l\left(D,\Sigma_{\row},\Sigma_{\col}\right)\\
 & = & m^{2}p-mp\log\left|\Sigma_{\row}^{-1}\right|-mp\log\left|\Sigma_{\col}^{-1}\right|+m^{2}\sum\log\left|{\rm tr}\left(Y_{i}\Sigma_{\col}^{-1}Y_{i}^{t}\Sigma_{\row}^{-1}\right)\right|+c.\end{eqnarray*}
There are two nonidentifiabilities in the scaled log likelihood given
by $l\left(D,\Sigma_{\row},\Sigma_{\col}\right)=l\left(aD,b\Sigma_{\row},\frac{1}{ab}\Sigma_{\col}\right)$
for $a,b>0$ and so we restrict our domain to consider minimization
the of $l\left(D,\Sigma_{\row},\Sigma_{\col}\right)$ over $\mathbb{R}_{+}^{p}\times\mathcal{S}_{2}\times\mathcal{S}_{2}$
where $\mathcal{S}_{2}$ is the bounded subset of $\mathcal{S}_{+}^{m}$
of positive definite matrices whose largest eigenvalue is 1. This
makes the model identifiable. Since minimization of $l\left(D,\Sigma_{\row},\Sigma_{\col}\right)$
is equivalent to minimization of $g\left(\Sigma_{\row}^{-1},\Sigma_{\col}^{-1}\right)$,
we restrict minimizing $g\left(\Sigma_{\row}^{-1},\Sigma_{\col}^{-1}\right)$
to $\mathcal{S}_{2}\times\mathcal{S}_{2}$. The continuity of $g$
on $\mathcal{S}_{+}^{m}\times\mathcal{S}_{+}^{m}\supset\mathcal{S}_{2}\times\mathcal{S}_{2}$
guarantees that it will attain a minimum on $\mathcal{S}_{2}\times\mathcal{S}_{2}$
as long as for $\overline{\Sigma_{\row}^{-1}},\overline{\Sigma_{\col}^{-1}}\in\overline{\mathcal{S}_{2}}$
\begin{eqnarray*}
\lim_{\Sigma_{\row}^{-1}\rightarrow\overline{\Sigma_{\row}^{-1}}}g\left(\Sigma_{\row}^{-1},\Sigma_{\col}^{-1}\right) & = & g\left(\overline{\Sigma_{\row}^{-1}},\Sigma_{\col}^{-1}\right)\\
\lim_{\Sigma_{\col}^{-1}\rightarrow\overline{\Sigma_{\col}^{-1}}}g\left(\Sigma_{\row}^{-1},\Sigma_{\col}^{-1}\right) & = & g\left(\Sigma_{\row}^{-1},\overline{\Sigma_{\col}^{-1}}\right)\\
\lim_{\Sigma_{\col}^{-1}\rightarrow\overline{\Sigma_{\col}^{-1}}}\lim_{\Sigma_{\row}^{-1}\rightarrow\overline{\Sigma_{\row}^{-1}}}g\left(\Sigma_{\row}^{-1},\Sigma_{\col}^{-1}\right) & = & g\left(\overline{\Sigma_{\row}^{-1}},\overline{\Sigma_{\col}^{-1}}\right).\end{eqnarray*}
All three conditions are met for positive definite boundary points,
so all that remains to show is that $g\left(\Sigma_{\row}^{-1},\Sigma_{\col}^{-1}\right)\rightarrow\infty$
when a subset of the eigenvalues of $\Sigma_{\row}^{-1}$ or $\Sigma_{\col}^{-1}$
approach 0 (behavior near positive semidefinite boundary points).
It is immediate that when the eigenvectors of $\Sigma_{\row}^{-1}$
do not match the left eigenvectors of $Y_{i}$ and the eigenvectors
of $\Sigma_{\col}^{-1}$ do not match the right eigenvectors of $Y_{i}$,
${\rm tr}\left(Y_{i}\Sigma_{\col}^{-1}Y_{i}^{t}\Sigma_{\row}^{-1}\right)\rightarrow c_{i}>0$,
thus the $\log{\rm tr}\left(\cdot\right)$ term converges to a finite
constant. As such, the behavior of $g\left(\Sigma_{\row}^{-1},\Sigma_{\col}^{-1}\right)$
when subsets of eigenvalues approach zero is completely governed by
the log determinant terms, both of which will converge to $+\infty$.
\qed